\documentclass{amsart}
\usepackage[english]{babel}
\usepackage{amsmath}
\usepackage{amsfonts}
\usepackage{amssymb}
\usepackage[all]{xy}
\usepackage{mathrsfs}
\usepackage{enumitem}
\usepackage{graphicx}
\usepackage{pgfplots}
\usepackage[linesnumbered,ruled]{algorithm2e}

\newtheorem{theorem}{Theorem}[section]
\newtheorem{lemma}[theorem]{Lemma}
\newtheorem{proposition}[theorem]{Proposition}

\newtheorem{corollary}[theorem]{Corollary}

\theoremstyle{definition}
\newtheorem{definition}[theorem]{Definition}
\newtheorem{example}[theorem]{Example}

\theoremstyle{remark}
\newtheorem{remark}[theorem]{Remark}

\newcommand{\C}{\mathbb{C}}
\newcommand{\Pj}{\mathbb{P}}
\newcommand{\Hy}{\mathbb{H}}
\newcommand{\heis}{\mathfrak{H}}
\newcommand{\R}{\mathbb{R}}

\newcommand{\N}{\mathbb{N}}

\begin{document}

\title[McMullen algorithm for complex schottky groups]{\sc{On McMullen's algorithm for the Hausdorff dimension of Complex Schottky Groups}  }

%
%
%

\author{Sergio Roma\~na}
\address{Instituto de Matemática, Universidade Federal do Rio de Janeiro, Cidade Universitária - Ilha do Fundão, Rio de Janeiro 21941-909, Brazil}
\email{sergiori@im.ufrj.br}

\author{Alejandro Ucan-Puc}
\address{ Institut de Mathématiques, Université Pierre et Marie Curie, 4 Place Jussieu, F-75252, Paris, France.}
\email{alejandro.ucan-puc@imj-prg.fr}


\subjclass{Primary 37M99, 37F99, 32M99 Secondary 30F45, 20H10, 57M60, 65Z05}

\begin{abstract}
We provide a generalization of the McMullen's algorithm to approximate the Hausdorff dimension of the limit set for convex-cocompact subgroups of isometries of the Complex Hyperbolic Plane.
\end{abstract}

\maketitle 

\section*{Introduction}

The Hausdorff dimension is a bi-Lipschitz invariant, and in the case of the Kleinian groups allows us to understand what kind of fractal spaces can be a limit set of Kleinian groups. In 1998, McMullen (\cite{Mcmullen1998}) proposed an algorithm to approximates the Hausdorff dimension of a set associated with a conformal dynamical system (as Julia sets or limit set of a geometrically finite Kleinian groups). Naively, McMullen's algorithm works as follows: 

\begin{enumerate}
\item[Step 1.] For a Markov partition of the dynamical system. We compute the transition matrix T using the data of the dynamical system. 
\item[Step 2.] We solve for $\alpha$ such that the spectral radius of $T^\alpha$ is 1. The matrix $T^\alpha$ is equal to the matrix where each entry is the entry of T powered by $\alpha.$ The power $\alpha$ is an approximation of the dimension of the conformal measure.
\item[Step 3.] We refine the Markov partition and do step 1 again. 
\end{enumerate}

For geometrically finite Kleinian groups, the conformal dynamical system conformed by the group and the associated Patterson-Sullivan measure (see \cite{Patterson1976},\cite{Sullivan1979}). The dimension of the associated Patterson-Sullivan measures coincides with the Hausdorff dimension of its limit set; from the previous, McMullen's algorithm works to approximate the Hausdorff dimension of the Kleinian group limit set. 

The Patterson-Sullivan theorems are not exclusive of the real hyperbolic geometry. Indeed this construction is generalized to different scenarios and geometries (see  \cite{Albuquerque1999}, \cite{Coornaert1993}, \cite{Coornaert1999}, \cite{Quint2006}), and all of them relates to the hyperbolic properties of the space. In particular, the Patterson-Sullivan theory is valid for the complex hyperbolic plane $\Hy^2_\C$ and subgroups of $Isom(\Hy^2_\C).$

In the complex hyperbolic setting, the discrete subgroups of $\mbox{\upshape{Isom}}(\Hy^2_\C)$ are a particular case of complex Kleinian groups, a term introduced by Seade and Verjovsky (see \cite{Cano2013}) as a generalization of the classic Kleinian group but for actions of $PSL(n+1,\C)$ in the complex projective n-space. 
The conformal dimension associated with a convex-cocompact subgroup of $\mbox{\upshape{Isom}}(\Hy^2_\C)$ also satisfies that its dimension coincides with the Hausdorff dimension of the limit set of the group in $\partial \Hy^2_\C.$ 

In the present, we propose a generalization of the McMullen algorithm. The algorithm we propose allows us to approximate the Hausdorff dimension of the limit set of a convex-cocompact discrete subgroup of $\mbox{\upshape{Isom}}(\Hy^2_\C),$ see, Theorems \ref{S2:Thm4:McMullenAlgorithm}, \ref{S2:Thm3:MarkovPartition} and  \ref{S2:Thm4:McMullenAlgorithm}. Also, we provide a pseudo-code for its computational implementation, see Appendix. We have to mention that the generalization doesn't follow directly from McMullen's paper. In the complex hyperbolic setting, the Patterson-Sullivan measure is defined using the visual metrics in $\partial\Hy^2_\C,$ meanwhile the Markov partition in Theorem \ref{S2:Thm2:MarkovPartitionCygan} is defined using the Cygan metric in $\partial\Hy^2_\C\setminus\{\infty\},$ this Markov partition definition follow the same spirit as the one in McMullen's work. Due to the work of Paulin-Hersonsky (\cite{Hersonsky2002a}), we obtain Theorem \ref{S2:Thm3:MarkovPartition} that implies that there exists a Markov partition for the boundary using the visual metric with the same properties of the one using the Cygan metric. 

Theorem \ref{S2:Thm4:McMullenAlgorithm}, states the order of the approximation of the generalized algorithm coincides with the one proposed by McMullen.

The paper is organized as follows: Section \ref{S1:Preliminaries} we settle the notation and construction that we will use in subsequent sections. Section \ref{S2:SchottkyGroups} deals with the construction of Schottky groups as subgroups of $\mbox{\upshape{Isom}}(\Hy^2_\C)$ generated by complex inversions, and we prove some general facts about them. Finally, in Section \ref{S3:McMullenAlgorithm} we construct the Markov partition associated with a Schottky group, we state the McMullen algorithm and some examples.

\section{Preliminaries}
\label{S1:Preliminaries}
\subsection{Complex Hyperbolic Plane}

Let $\C^{2,1}$ denote the vector space $\C^3$ with the Hermitian form 
\begin{equation}
\label{S1:Eq1:Hermitian}
\left\langle\left(\!\begin{array}{c}
z_1 \\ z_2 \\ 1
\end{array}\!\right)\left(\!\begin{array}{c}
w_1 \\ w_2 \\ 1
\end{array}\!\right)\right\rangle = \overline{z}_1 w_3+ \overline{z}_2 w_2 + \overline{z}_3 w_1.
\end{equation}

and let $\pi:\C^3\setminus\{0\}\rightarrow \Pj^2_\C$ the natural projection. The \emph{complex hyperbolic plane}, denoted by $\Hy^2_\C,$ is the image under $\pi$ of the set of negative vectors in $\C^{2,1}.$ The boundary of $\Hy^2_\C$ is defined as the image under $\pi$ of the set of null vectors and we will denoted by $\partial\Hy^2_\C.$ The complex hyperbolic plane equipped with the Bergman metric  (see \cite{Goldman1999}) is a complete K\"ahler manifold of constant holomorphic sectional curvature -1.

Let $U(2,1)$ denote the unitary group associated to the Hermitian form \ref{S1:Eq1:Hermitian}. The set of \emph{holomorphic isometries} of $\Hy^2_\C$ is the projective unitary group $PU(2,1)$ and the full group of isometries $\mbox{Iso}(\Hy^2_\C)$ is generated by $PU(2,1)$ and the complex conjugation (see \cite{Goldman1999}, \cite{Parker2003}).

\subsubsection{Horospherical coordinates and Heisenberg structure.}

Let $z\in\partial\Hy^2_\C,$ given by 
\[z=\left[\begin{array}{c} z_1 \\ z_2 \\ 1 \end{array}\right],\]
then $z$ has to satisfy $2\Re(z_1)+|z_2|^2=0.$ So writing $z_2=\sqrt{2}\zeta$ and $z_1=-|\zeta|^2+iv,$ we ensure that the previous equation is satisfied and we obtain an identification of $\partial\Hy^2_\C$ with the one point compactification of $\C\times\R.$ The boundary has a structure of Heisenberg group (see \cite{Goldman1999},\cite{Parker2003}) with group operation 
\[(\zeta,v)\ast (\xi,t)=(\zeta+\xi,v+t+2\Im(\overline{\zeta}\xi))\]

we will denote by $\heis$ the set $\C\times\R$ with the Heisenberg structure.

For a fixed $u\in\R_+,$ and consider the points $z\in\Hy^2_\C$ such that $\langle z,z\rangle=-2u,$ these points have the form
\begin{equation}
\label{S1:Eq2:HorosphericalCoordinates}
\left[\begin{array}{c} -|\zeta|^2-u+iv \\ \sqrt{2} \zeta \\1 \end{array}\right],
\end{equation}
so a point $z$ correspond to a triplet $(\zeta,v,u)\in\heis\times\R_+,$ known as \emph{horospherical coordinates}. Let us denote by $H_u$ the locus of points in $\Hy^2_\C$ such that $\langle z,z\rangle=-2u,$ this set is called \emph{horosphere of weight $u$,} $H_u$ carries the Heisenberg group structure. It will be useful to identify the finite points of $\partial\Hy^2_\C\setminus\{\infty\}$ with $H_0.$ We will define the \emph{horoball of weight u} as the union of the $H_t$ for $t\geq u,$ so the horoball of weight zero is equal to $\Hy^2_\C.$

The \emph{Heisenberg norm} is given by $|(\zeta,v)|_0=\left||\zeta|^2+iv|\right|^{\frac{1}{2}},$ this norm induces a metric on $\heis,$ called the \emph{Cygan metric,} defined by 

\begin{equation}
\label{S1:Eq3:CyganMetric}
d_{Cyg}((\zeta_1,v_1),(\zeta_2,v_2))=|(\zeta_1-\zeta_2,-v_1-+v_2-2\Im(\overline{\zeta}_2\zeta_1))|_0.
\end{equation}

The previous metric can be extended to the whole $\overline{\Hy^2_\C}\setminus\{\infty\}$ as an incomplete metric as follows
\[d_{Cyg}((\zeta_1,u_1,v_1),(\zeta_2,u_2,v_2))=||\zeta_1-\zeta_2|^2+|u_1-u_2|+ i(-v_1+v_2-2\Im(\overline{\zeta}_2\zeta_1))|^{\frac{1}{2}}.\] 

We will denote the \emph{Cygan balls} by $B((\xi,t),r)$ at the set of points $(\zeta,v,u)\in\heis\times\R_+$ such that $d_{Cyg}((\xi,t,0),(\zeta,v,u))=r,$ and we will call \emph{Cygan sphere,} $S((\xi,t),r),$ at the intersection $B((\xi,t),r)\cap\partial\Hy^2_\C.$

The Heisenberg group acts on itself by
\begin{itemize}
\item \emph{Heisenberg translations:} $T_{(\zeta,v)}(\xi,t)=(\zeta+\xi,v+t+2\Im(\overline{\zeta}\xi)).$
\item \emph{Complex dilatations:} $d_{\lambda}(\zeta,v)=(\lambda z, |\lambda|^2 v).$
\end{itemize}

The first ones are isometries of the Cygan metric, and for the complex dilations these are isometries if and only if $|\lambda|=1.$

We will call an embedded copy of $\Hy^1_\C$ on $\Hy^2_\C$ a \emph{complex geodesic,} which is a totally geodesic submanifold of dimension 2 and constant sectional curvature equal -1. The intersection of a complex geodesic $L$ with $\partial\Hy^2_\C$ is called \emph{chain} and denoted by $\partial L.$ Chains passing through $\infty$ are called \emph{vertical or infinite chains} and otherwise is called \emph{finite chain.} For every finite chain $C,$ there exist a translation and a complex dilatation such that $C$ is image of $\mathbb{S}^1\times\{0\}\subset\heis$ under these mappings. For every complex geodesic $L,$ there is a unique element of $PU(2,1)$ whose fixed-point set is $L,$ these maps are called \emph{complex reflections.} The action of a complex reflection on the boundary preserves the chain associated to the complex geodesic.

The following lemma gives a description of the Cygan metric by the action of an element of $PU(2,1).$
\begin{lemma}
\label{S1:Lmm1:CyganDist}
Let $g\in PU(2,1)$ that does not fix $\infty.$ Then there exists a positive number $r_g$ depending only in $g$ such that for all $z,w\in\partial\Hy^2_\C\setminus\{\infty,g^{-1}(\infty)\},$ we have:
\begin{enumerate}[label=\roman*]
\item $d_{Cyg}(g(z),g(w))=\frac{r^2_g d_{Cyg}(z,w)}{d_{Cyg}(z,g^{-1}(\infty))d_{Cyg}(w,g^{-1}(\infty))},$
\item $d_{Cyg}(g(z),g(\infty))=\frac{r^2_g}{d_{Cyg}(z,g^{-1}(\infty))}.$
\end{enumerate}
\end{lemma}

An important consequence of the previous lemma is that $g$ sends the Cygan sphere of radius $r_g$ and centre $g^{-1}(\infty)$ to the Cygan sphere of radius $r_g$ and center $g(\infty).$ So as an analogous of real hyperbolic geometry, it's defined the \emph{isometric sphere} of $g$ as the Cygan sphere of radius $r_g$ and center $g^{-1}(\infty).$

\begin{proposition}[\cite{Parker1994}]
\label{S1:Prp1:ImageCyganSpheres}
Let $h$ be an element of $PU(2,1)$ not fixing $\infty.$ Then the Cygan sphere of radius $r$ and centre $h^{-1}(\infty)$ is mapped by $h$ into the Cygan sphere of radius $r^2_h/r$ and centered at $h(\infty).$
\end{proposition}

\subsection{Patterson-Sullivan Measures for subgroups of $PU(2,1)$}

A discrete subgroup $\Gamma$ of $PU(2,1)$ is called a \emph{complex Kleinian group} (see \cite{Cano2013}), these groups gives a partition of $\overline{\Hy^2_\C}$ in two invariant sets:

\begin{itemize}
\item The first is the \emph{Chen-Greenberg limit set} who is the closure of the cluster points of $\Gamma$ orbits of points in $\Hy^2_\C,$ is a subset of $\partial \Hy^2_\C$ and it is denoted by $\Lambda_{CG}(\Gamma).$

\item The second is the \emph{discontinuity region,} denoted by $\Omega(\Gamma),$ given by $\overline{\Hy^2_\C}\setminus \Lambda_{CG}(\Gamma).$
\end{itemize}

In the classical setting, for a convex-cocompact subgroup $\Gamma$ of $PSL(2,\C),$ there is a density $\mu$ associated to $\Gamma,$ such that it is invariant by the group (see \cite{Sullivan1984}). The existence of this density is not exclusively for  $\mathbb{H}^3$ and convex cocompact subgroups of $PSL(2,C).$ This theory can be extended for hyperbolic manifolds or non-postively curved spaces (\cite{Coornaert1999},\cite{Quint2006},\cite{Kapovich2002}). 

\begin{definition}
 Let $X$ be a complete simply connected Riemannian manifold with non-positive curvature and $\Gamma$ a discrete subgroup of $Isom(X),$ a map $x\mapsto \nu_x$ from $X$ to the set of Radom measures on $\partial X$ wich:
\begin{enumerate}[label=\roman*]
\item It is $\Gamma-$equivariant.
\item If $x,y \in X,$ then $\nu_x$ and $\nu_y$ are equivalent.
\item For every $\zeta\in\partial X,$ we have \[\frac{d \nu_y}{d\nu_x}(\zeta)=e^{-\beta b_{\zeta}(y,x)},\] where $b_\zeta(y,x)$ is the Busemann function.
\end{enumerate} We will say that this map is a $\Gamma-$\emph{conformal density} of dimension $\beta.$ 
\end{definition}

\begin{theorem}[\cite{Coornaert1999}]
\label{S2:Thm3:ConformalDensityGamma}
Let $(X,d)$ a proper hyperbolic space and $\Gamma$ a group acting on $X$ by isometries properly discontinuous and convex-cocompactly. Let $d_a$ the visual metric on $\partial X$ and $\Lambda(\Gamma)$ the limit set of $\Gamma.$ Put $D=\delta_a(\Gamma)$ the exponent of growth of $\Gamma$ and $H=H^D$ the $D-$Hausdorff measure on $\Lambda(\Gamma)$ with respect to $d_a.$ Then
\begin{enumerate}[label=\roman*]
\item $\delta_a(\Gamma)$ is the Hausdorff dimension of $\Lambda(\Gamma).$
\item $H$ on $\Lambda(\Gamma)$ is a $\Gamma-$conformal density of dimension $\delta_a(\Gamma).$
\item If $\mu$ is a $\Gamma-$conformal measure of dimension $D'$ with support $\Lambda(\Gamma)$ then $D'=D=\delta_a(\Gamma)$ and $\mu$ and $H$ are equivariant.
\end{enumerate}
\end{theorem}


For discrete subgroups of $PU(2,1),$ the convex cocompact groups are characterized by the type of its limit point.

\begin{definition}
A sequence $(x_i)_{i\in\N}$ of different points in $\Hy^2_\C$ is said to converge to a point $\zeta\in\partial\Hy^2_\C$ \emph{conically} if there exists a geodesic ray $x\zeta$ in $\Hy^2_\C$ and a constant $R<\infty$ such that: \[d(x_i,x\zeta)\leq R\] for all $i$ and $\displaystyle{\lim_{i\to\infty}}x_i=\zeta.$ We will denote $\Lambda^c_\Gamma$ the set of conical limit points. 
\end{definition}

\begin{theorem}[\cite{Bowditch1995}]
A discrete group $G<PU(2,1)$ is convex-cocompact if and only if every limit point of $\Gamma$ is conical.
\end{theorem}

\subsection{Anosov Representations into $PU(2,1)$.}

Let $G$ be a semi-simple Lie group with finite center and Lie algebra $\mathfrak{g}.$ Let $K$ be a maximal compact subgroup of $G$ and let $\tau$ be the Cartan involution on $\mathfrak{g}$ whose fixed point set is the Lie algebra of $K.$ Let $\mathfrak{a}$ be a maximal abelian subspace contained in $\{v\in \mathfrak{g}:\tau v=-v\}.$ 

For $a\in \mathfrak{a},$ let $M$ be the connected component of the centralizer of $\exp(a)$ which contains the indetity, and let $\mathfrak{m}$ denote its Lie algebra. Let $\mathfrak{g}_{\lambda}$ be the eigenspace of the action of $a$ on $\mathfrak{g}$ with eigenvalue $\lambda$ and consider 

\[\mathfrak{n}^+=\bigoplus_{\lambda>0} \mathfrak{g}_\lambda, \quad \mathfrak{n}^-=\bigoplus_{\lambda<0} \mathfrak{g}_\lambda,\]

so that $\mathfrak{g}=\mathfrak{m}\oplus \mathfrak{n}^+\oplus \mathfrak{n}^-.$ 

Let $P^{\pm}$ be the Lie subgroups of $G$ which normalize the Lie algebras $\mathfrak{p}^\pm=\mathfrak{m}\oplus \mathfrak{n}^\pm,$ and consider the associated flag manifolds $G/P^\pm.$ We will say that the pair $([g_1],[g_2])\in G/P^+\times G/P^-$ is \emph{transverse,} if the intersection $g_1 P^+ g_1^{-1}\cap g_2 P^- g_2^{-1}$ is conjugate to $M.$

Let $\varrho:\Gamma\to G$ be a representation of a word hyperbolic group $\Gamma,$ and let $\xi^\pm:\partial_\infty\Gamma\to G/P^\pm$ be two continuous $\varrho-$equivariant maps. We can define the transversity for $(\xi^+,\xi^-)$ if for any pair of distinct points $x,y\in\partial_\infty\Gamma$ the image is transverse pair. 

\begin{definition}[Definition 2.9, \cite{Bridgeman2015}]
Suppose that $G$ is a semi-simple Lie group with finite center, $P^+$ a parabolic subgroup of $G$ and $\Gamma$ is a word-hyperbolic group. A representation $\varrho:\Gamma\to G$ is said to be $(G,P^+)-$Anosov if the following holds:
\begin{enumerate}[label=\roman*] 
\item If there exists two  transverse $\varrho-$equivariant maps, $\xi^\pm:\partial_\infty \Gamma\to G/P^\pm.$
\item There exists two bundles over the unitary tangent bundle of $\Gamma,$ $\mathcal{E}_\varrho^\pm$ such that the geodesic flow on $\partial_\infty\Gamma$ is contractive in $\mathcal{E}_\varrho^+$ and the inverse geodesic flow is contractive in $\mathcal{E}_\varrho^-.$
\end{enumerate}
\end{definition}

\begin{theorem}[Thereom 5.15 in \cite{Guichard2012}]
Let $G$ be a Lie group of real rank 1. Let $\varrho:\Gamma\to G$ be a representation. The following are equivalent:

\begin{enumerate}[label=\roman*]
\item $\varrho$ is Anosov.
\item There exists $\xi:\partial_\infty \Gamma$ to $G/P$ a continuous, injective and equivariant map.
\item  $\varrho$ is a quasi-isometric embedding.
\item $\varrho(\Gamma)$ is convex cocompact.
\end{enumerate}

\begin{remark}
The previous theorem implies that if $\varrho:\Gamma\to PU(2,1)$ whose image is convex cocompact then $\varrho$ is an Anasov representation; even more, we can say that limit map $\xi$ have image in $\partial \Hy^2_\C$ and $\xi(\partial_\infty\Gamma)=\Lambda_{CG}(\varrho(\Gamma)).$
\end{remark}
\end{theorem}

The following theorem is an interesting result of Anosov representations that deals with some analytic properties implied by the Anosov definition.
\begin{theorem}[\cite{Bridgeman2015}]
\label{AnalyticLimitMap}
Let $G$ be a real semi-simple Lie group with finite center and let $P$ be a parabolic subgroup of $G.$ Let $\{\varrho_u\}_{u\in D}$ be a real analytic family of homomirphism of $\Gamma$ into $G$ parametrized by a real disk $D$ about 0. If $\varrho_0$ is a $(G,P)-$Anosov homomorphism with limit map $\xi_0:\partial_\infty\Gamma\to G/P,$ then there exists a sub-disk $D_0\subset D$ around 0, $alpha>0$ and a unique continuous map 
\begin{equation}
\xi:D_0\times \partial_\infty\Gamma\to G/P
\end{equation}
so that $\xi(0,\cdot)=\xi_0(\cdot)$ with the following properties:
\begin{enumerate}[label=\roman*]
\item If $u\in D_0,$ then $\varrho_u$ is a $(G,P)-$Anosov homomorphism with limit map $\xi_u:\partial\infty\Gamma \to G/P$ given by $\xi_u(\cdot)=\xi(u,\cdot).$
\item If $x\in\partial\infty\Gamma,$ then $\xi_x:D_0\to G/P$ given by $\xi_x=\xi(\cdot,x)$ is real analytic.
\item The map from $D_0$ to $C^\alpha \left(\partial_\infty \Gamma,G/P  \right),$ the set of $\alpha-$Hölder maps, given by $u\to \xi_u$ is real analytic.
\end{enumerate}
\end{theorem}

\section{Schottky groups in $PU(2,1).$}
\label{S2:SchottkyGroups}
Let $C_0$ denote the chain $\mathbb{S}^1\times \{0\}$ and $\iota:\heis\rightarrow\heis$ given by
\begin{equation}
\label{S2:Eq1:KoranyiInv}
\iota(\zeta,v) =\left(\frac{-\zeta}{|\zeta|^2-iv},\frac{-v}{|\zeta|^4+v^2}\right),
\end{equation}
a straight computation give us that $\iota(C_0)=C_0,$ the map in (\ref{S2:Eq1:KoranyiInv}) is called \emph{Koranyi inversion.} Using the identification (\ref{S1:Eq2:HorosphericalCoordinates}), $\iota$ has the matrix form

\begin{equation}
\label{S2:Eq2:MatrixKoranyi}
\iota=\left(\!\begin{array}{ccc}
0 & 0 & 1 \\ 0 & -1 & 0 \\ 1 & 0 & 0
\end{array}\!\right)
\end{equation}

where $\iota\in PU(2,1),$ even more $\iota$ is a complex reflection. In particular, the isometric sphere of $\iota$ is the Cygan unit sphere centered at $(0,0).$ If we take a finite chain in $\heis,$ then the complex reflection that defines is equal to 
\begin{equation}
\label{S2:Eq3:ComplexReflection}
\iota_C=D_\lambda T_{(\xi,t)}\iota T_{(-\xi,-t)}D_{\lambda^{-1}}
\end{equation}

where $C=D_{\lambda}T_{(\xi,t)}(\mathbb{S}^1\times\{0\}).$ The isometric sphere of $\iota_C$ is the Cygan sphere of center $(\xi,t)$ and center $|\lambda|.$

\begin{lemma}
\label{S2:Lmm1:JacobianCR}
Let $C$ a finite chain of center $(\xi,t)$ and center $|\lambda|,$ and let $\iota_C$ the induced complex reflection. For $(\zeta_0,v_0)\in\partial \Hy^2_\C\setminus\{(\xi,t),\infty\},$ we have that 
\begin{equation}
\label{S2:Eq3:JacobianCR}
\left|\det\left(\left.\frac{\partial \iota_C}{\partial(\zeta,v)}\right|_{(\zeta_0,v_0)}\right)\right| = \frac{|\lambda|^4}{(d_{Cyg}((\zeta_0,v_0),(\xi,t)))^4},
\end{equation}
where $\frac{\partial \iota_C}{\partial(\zeta,v)}$ denotes the Jacobian matrix of $\iota_C$ as a real valued function.
\end{lemma}

\begin{proof}
First, we will prove it for the Koranyi inversion, who in real variables is of the form

\begin{equation}
\iota(x,y,z)=\left(\frac{x(x^2+y^2)+yz}{(x^2+y^2)^2+z^2},\frac{-xz-y(x^2+y^2)}{(x^2+y^2)^2+z^2},\frac{-z}{(x^2+y^2)^2+z^2}\right),
\end{equation}

and a straight computation gives that 
\[\det\left(\left.\frac{\partial \iota}{\partial(x,y,z)}\right|_{(x_0,y_0,z_0)}\right)=\frac{1}{(x_0^2+y_0^2)^2+z_0^2}\]

which is what we were claiming. For the general case, we have to note that a general complex reflection is given by $\iota_{C}=T_{(\xi,t)}D_\lambda \iota D_{\lambda^{-1}}T_{(-\xi,-t)}.$ Straights computations gives that
\[\det\left(\frac{\partial T_{(\xi,t)}}{\partial(\zeta,v)}\right)=1 \hspace{0.5cm}\det\left(\frac{\partial D_\lambda}{\partial(\zeta,v})\right)=|\lambda|^2\] and this determinants doesn't depend on the evaluation point. By the chain rule, we have the claim.
\end{proof}

In classical real hyperbolic geometry, for a M\"obius map $g\in PSL(2,\C)$ of the form the image under $g$ of a small circle centered at a point $z\in\hat{\C}$ is "distorted" by a factor approximate to $|f'(z)|,$ for smaller circles centered at $z$ more exact is the distortion factor.

\begin{lemma}
\label{S2:Lmm2:DistortionFactor}
Let $\iota_C\in PU(2,1)$ a complex reflection and $z\in \partial\Hy^2_\C\setminus\{\infty,\iota_C(\infty)\}.$ A small Cygan sphere centered in $z$ is distorted by $\iota_C$ a factor approximate
\begin{equation}
\label{S2:Eq4:DistortionFactor}
\sqrt{\left|\det\left(\left.\frac{\partial \iota_C}{\partial (\zeta,v)}\right|_{z}\right)\right|}.
\end{equation}
For smaller Cygan sphere more accurate the approximation.
\end{lemma}
\begin{proof}
It will be sufficient to prove it for the Koranyi inversion, the general case is a consequence of the lemma \ref{S2:Lmm1:JacobianCR} and the chain rule.

Let $(\zeta_0,v_0)\in\partial \Hy^2_\C \setminus\{0,\infty\}$ and let $S_r((\zeta_0,v_0))$ the Cygan sphere of radius $r$ and center $(\zeta_0,v_0).$ Let us take $(\zeta,v)\in S_r(\zeta_0,v_0),$ by lemma \ref{S1:Lmm1:CyganDist}, we know that for $(\zeta,v),(\zeta_0,v_0)$ is satisfied

\[d_{Cyg}(\iota(\zeta_0,v_0),\iota(\zeta,v))=\frac{ r}{d_{Cyg}((\zeta_0,v_0),(0,0))\rho_0((\zeta,v),(0,0))},\] when $(\zeta,v)$ is close enough $(\zeta_0,v_0)$ the previous equation is approximely to what we are claiming. 
\end{proof} 

\begin{definition}[\cite{Mumford2002}]
\label{Schottkygroup}
Let $\mathscr{C}=\{C_i\}_{i=1}^k$ a finite family of finite chains and $\{\iota_i\}_{i=1}^k$ the associated complex reflections. Let us assume that:
\begin{enumerate}[label=\roman*]
\item  the isometric spheres $S_{\iota_i}$ are disjoint in $\heis,$
\item the Cygan balls bounded by the isometric spheres $\{S_{\iota_j}\}$ are disjoint.
\end{enumerate}

 then $\Gamma(\mathcal{C})=\langle \{\iota_i\}_{i=1}^k\rangle$ is a \emph{Schottky group}.

\end{definition}

\begin{remark}
 We call this groups Schottky because restricted to the closure of the complex hyperbolic plane they have the Ping-Pong dynamics, but if we consider the action of the group to the whole projective space we have to call this group a \emph{Schottky-like} group (see \cite{Conze2000}). For the special case of three reflections these groups are called \emph{triangular groups} (see \cite{Goldman1992}, \cite{Schwartz2005}).
 \end{remark}
 
 \begin{proposition}
 Let $\mathscr{C}$ be a finite collection of finite chains as in the previous definition. The group $\Gamma(\mathscr{C})$ is a convex cocompact subgroup of $PU(2,1).$
 \end{proposition}
 \begin{proof}
 Let $\zeta\in\Lambda_{CG}(\Gamma(\mathscr{C}))$ be a limit point, and let $\alpha$ be any geodesic ray that converges to $\zeta$ and $x\in\Hy^2_\C$ any point. From the definition of $\Gamma(\mathscr{C}),$ we have that $\zeta\in\heis$ is the limit of a sequence of nested isometric spheres in $\heis$ of elements in $\Gamma(\mathscr{C}).$ 
 
 Let $\gamma:[0,\infty)\to\Hy^2_\C$ be any geodesic ray in $\Hy^2_\C,$ such that $\displaystyle{\lim_{t\to\infty}}\gamma(t)=\zeta.$ Let $N_A(\gamma(t))$ be a $A-$neighborhood of the geodesic ray. From the previous paragraph, we can assure that the Cygan balls bounded by the sequence of isometric sphere of elements in $\Gamma(\mathscr{C})$ intersects $N_A(\gamma),$ even more we can assure that for small enough $A,$ the intersection of $N_A(\gamma)$ with the cygan ball bounded by a isometric sphere, is still bounded by the isometric sphere.

Let $(x_j)_{j\in\N}\subset\Hy^2_\C$ any sequence that converge to $\zeta,$ and from these we can say that exists $N\in \N$ such that for $j\geq N,$ then $x_j$ belongs to a Cygan ball bounded by an isometric sphere of an element of $\Gamma(\mathscr{C}),$ even more, we can assure that these isometric sphere are images under a sequence of the generators. Since this generators are finite collection of complex reflections and from Lemma \ref{S2:Lmm2:DistortionFactor} and Proposition \ref{S1:Lmm1:CyganDist}, we can assure that the distance from the points of the sequence to the geodesic ray is bounded.
 \end{proof}

\begin{corollary}
Let $\Gamma$ be word-hyperbolic free group generated by $\{g_j\}_{j\in \N}$ and let $\mathcal{C}=\{C_j\}_{j\in\N}$ be a finite collection of finite chains such that $\Gamma(\mathcal{C})$ is a Schottky group of $PU(2,1)$ (see Definition \ref{Schottkygroup}). The representation $\varrho:\Gamma\to \Gamma(\mathcal{C})$ given by $\varrho(g_j)=\iota_j$ where $\iota_j$ is the inversion induced by the chain $C_j,$ is an Anosov representation. 
\end{corollary}

\begin{remark}
Something that we can say about the Schottky groups defined as in Definition \ref{Schottkygroup} from the previous Corollary is that the limit map that sends the boundary of $\Gamma$ to the limit set of $\Gamma(\mathcal{C})$ is analytic and by Theorem \ref{S2:Thm3:ConformalDensityGamma}, the Hausdorff dimension of $\Lambda_{CG}(\Gamma(\mathcal{C})$ is analytic too.
\end{remark}

\section{The McMullen Algorithm and experimentations}
\label{S3:McMullenAlgorithm}

The McMullen's algorithm for computing the Hausdorff dimension of a Schottky acting on the three dimensional real hyperbolic space (see \cite{Mcmullen1998}) used the \emph{eigenvalue algorithm} with a Markov partition to approximate the dimension of the unique measure associated to the group (see \cite{Sullivan1979}). In this section we propose a Markov partition that works to compute the dimension of the conformal measure associated to a Schottky group.

Following the ideas of \cite{Anderson1997}, we will construct a Markov partition associated to a complex Schottky group. So, let $\Gamma< PU(2,1)$ discrete, $\{P_i\}_{i=1}^k$ a finite collection of domains in $\heis$ such that $int(P_i)\cap int(P_j)=\varnothing$ for $i\neq j,$ and let $P_0=\overline{\heis\setminus \bigcup_{i=1}^kP_i}$ and has finitely many components, it is easy to note that $\heis=P_0\cup\cdots\cup P_k.$

\begin{enumerate}
\item[(M0)] $P_0$ contains the clousure of a fundamental domain for $\Gamma.$
\item[(M1)] $\partial P_j\cap \Lambda(\Gamma)$ is finite for every $j=1,\cdots,k.$
\end{enumerate} 

There is a map $f:\heis\rightarrow\heis$

\begin{enumerate}
\item[(M2)] There are some $\gamma_j\in \Gamma$ such that $f|_{P_j}=\gamma_j|_{P_j}$ for $1\leq j\leq k$ and $f|_{P_0}=id.$
\item[(M3)] $f(P_i)=P_{j_1}\cup\cdots\cup P_{j_n}$ for some $j_1,\cdots,j_n\in\{0,\cdots,k\}.$

\end{enumerate}
For $x\in\heis,$ let $({j_0},{j_1},\cdots)$ such that $x\in P_{j_0},\, f(x)\in P_{j_1},\cdots.$ A finite sequence $(j_0,\cdots, j_m)$ is called \emph{admisible} if $f(P_{j_l})\supseteq P_{j_{l+1}}$ for every $0\leq l\leq m-1,$ and define $P(j_0,\cdots,j_m)=\bigcap_{i=1}^m f^{-1}(P_{j_i}).$

\begin{enumerate}
\item[(M4)] If for every sequence $(j_0,\cdots,j_l,\cdots),$ then 
\[Diam_{Euclid}(P(j_0,\cdots,j_n))\xymatrix{\ar[r]_{n\rightarrow \infty}&}0.\]
\item[(M5)] If there exists $N\in \mathbb{N}$ and $\beta>1$ such that 
\[\left|\det\left(\frac{\partial^N f}{\partial^N (\zeta,v)}(x)\right)\right|>\beta,\] for every $x\in P(j_0,\cdots,j_l)$ where $(j_0,\cdots,j_l)$ is admisible.
\end{enumerate}

A convex-cocompact complex kleinian group in $PU(2,1)$ that satisfies (M0)-(M5) is said to be that have the \emph{expanding Markov property} for $\mathscr{P}=\{P_j\}_{j=0}^k$ and $f:\heis\rightarrow\heis.$

\begin{theorem}
\label{S2:Thm2:MarkovPartitionCygan}
Let $\mathscr{C}=\{C_j\}_{j=1}^d$ a finite collection of finite chains such that $\Gamma(\mathscr{C})$ is a complex kleinian group. Let $\mathscr{D}=\{D_i\}_{i=0}^d$ such that $D_i=int(S_i)$ where $S_i$ is the isometric sphere of the complex reflection induced by $C_i$ and $D_0=\overline{\heis\setminus\bigcup_{i=1}^d D_i},$ and $f:\heis\rightarrow\heis$ given by $f|_{D_j}=\iota_{C_j}|_{D_j}$ for $j=1,\cdots,d$ and $f|_{D_0}=id.$ Then $(\Gamma(\mathscr{C}),\mathscr{D},f)$ has the expansive Markov property.
\end{theorem}
\begin{proof}
By construction of the Schottky group $\Gamma(\mathscr{C}),$ we can assure that $P_0$ contains a fundamental domain, so $\Gamma(\mathscr{C})$ satisfies (M0), and by construction $\partial D_j\cap \Lambda_{CG}(\Gamma(s\mathscr{C}))$ has a finite number of points, then $\Gamma(\mathscr{C})$ satisfies (M1). By hypothesis, it's satisfied (M2). Since, every $\iota_j$ satisfies that $Int(S_j)$ is mapped to $Ext(S_j),$ then $\Gamma(\mathscr{C})$ satisfies (M3). By proposition \ref{S1:Prp1:ImageCyganSpheres} and \ref{S2:Lmm2:DistortionFactor}, we can assure that (M4) happens, and finally by Proposition \ref{S2:Lmm2:DistortionFactor} we have that $f$ has the expansive property for every point inside the isometric sphere. So we can conclude that $(\Gamma(\mathscr{C}),\mathscr{D},f)$ has the expansive Markov property.

\end{proof}

Notice that the Markov partition is given for the Cygan metric in $\heis=\partial\Hy^2\setminus\{\infty\}$ and the conformal density is defined for the Hausdorff measure using the Gromov metric in $\partial\Hy^2_\C.$ The next lemma implies that the Markov partition that we have works as a Markov partition in the Gromov metric.

Let $\xi^+$ be the positive end of the geodesic that pass trough $o\in\Hy^2_\C$ and $\infty,$ since the action of $\heis$ on $\partial\Hy^2_\C$ is transitive outside $\infty,$ we can obtain a map $\phi:\heis\rightarrow\partial\Hy^2_\C$ given by $\phi(s)=s(\xi^+).$

\begin{lemma}[\cite{Hersonsky2002a}, \cite{Dufloux2017}]
 Let $\heis$ be endowed with the Cygan metric and $\partial\Hy^2_\C$ endowed with the Gromov metric, then the map $\phi:\heis\rightarrow\partial\Hy^2_\C\setminus\{\infty\},$ defined previously, is bilipschitz.
\end{lemma}

So, by Theorems \ref{S2:Thm2:MarkovPartitionCygan}, \ref{S2:Thm3:ConformalDensityGamma} and the previous Lemma, we can assure that for a Schottky group $\Gamma(\mathscr{C})$ the eigenvalue algorithm applied to the Markov partition $(\mathscr{D},\{\iota_{C_i}\}_{C_i\in\mathscr{C}}),$ gives an approximation of the dimension of the $\Gamma(\mathscr{C})-$conformal density; moreover, an approximation of the Hausdorff dimension of the Chen-Greenberg limit set $\Lambda_{CG}(\Gamma(\mathscr{C})).$ The following theorem is a direct consequence of the previous paragraph.

\begin{theorem}
\label{S2:Thm3:MarkovPartition}
 Let $\Gamma\subset\mbox{\upshape{PU}}(2,1)$ a Schottky group. There exists a partition contained in the Markov partition of $\Gamma$ such that it is Markov for a visual metric on $\partial\Hy^2_\C\setminus\{\infty\}.$
\end{theorem}

The following theorem is similar to the Corollary 3.4 in \cite{Mcmullen1998}.

\begin{theorem}
\label{S2:Thm4:McMullenAlgorithm}
For a disjoint family of finite chains,  $dim_H(\Lambda_{CG}(\Gamma(\mathscr{C})))$ can be computing by applying the eigenvalue algorithm to the Markov partition given by the isometric spheres of the generating elements.
\end{theorem}

The following theorem due to McMullen in \cite{Mcmullen1998} implies that for a given dynamical system the order of approximation of the digits of the measure dimension is lineal depending on the step of refinement in the Markov partition.

\begin{theorem}[Theorem 2.2 in \cite{Mcmullen1998}]
Let $\mathscr{P}$ a expanding Markov partition for a conformal dynamical sistem $\mathscr{F}$ with invariant density $\mu$ of dimension $\delta.$ Then \[\alpha(\mathcal{R}^n(\mathscr{P}))\rightarrow \delta\] as $n\rightarrow\infty,$ where $\mathcal{R}^n(\mathscr{P})$ denotes the $n^{th}-$refinement of $\mathscr{P}.$ At most $O(N)$ refinements are required to compute $\delta$ to $N$ digits of accuracy. 
\end{theorem}

Since we have that there exists a Bi-Lipzchits map between the Boundary with the Gromov metric and the Cygan metric, the Hausdorff dimensions are equal. So the previous theorem is valid and direct.

\subsection{Examples}

\begin{example}[Symmetric $\theta-$Schottky Group]
 Let $0<\theta<\pi/3,$ and let $\mathscr{S}$ the configuration of three chains in $\heis$ with centers in $\C\times\{0\}$ and symmetric under rotation of $\pi/3$ around the $z$ vertex. These chains are parametrized by the following
 \begin{eqnarray*}
  t&\mapsto& \left(\frac{1}{\cos(\theta)}+\tan(\theta)e^{it},\frac{2\sin(t)}{\cos(\theta)}\right)\\
  t&\mapsto& \left(\frac{w_1}{\cos(\theta)}+\tan(\theta)e^{it},\frac{\sqrt{3}\cos(t)-\sin(t)}{\cos(\theta)}\right)\\
  t&\mapsto& \left(\frac{w_2}{\cos(\theta)}+\tan(\theta)e^{it},\frac{-\sqrt{3}\cos(t)-\sin(t)}{\cos(\theta)}\right)
 \end{eqnarray*}
  where $\{1,w_1,w_2\}$ are the cubic roots of the unity.
  \begin{figure}[h]
  \centering
  \includegraphics[scale=0.55]{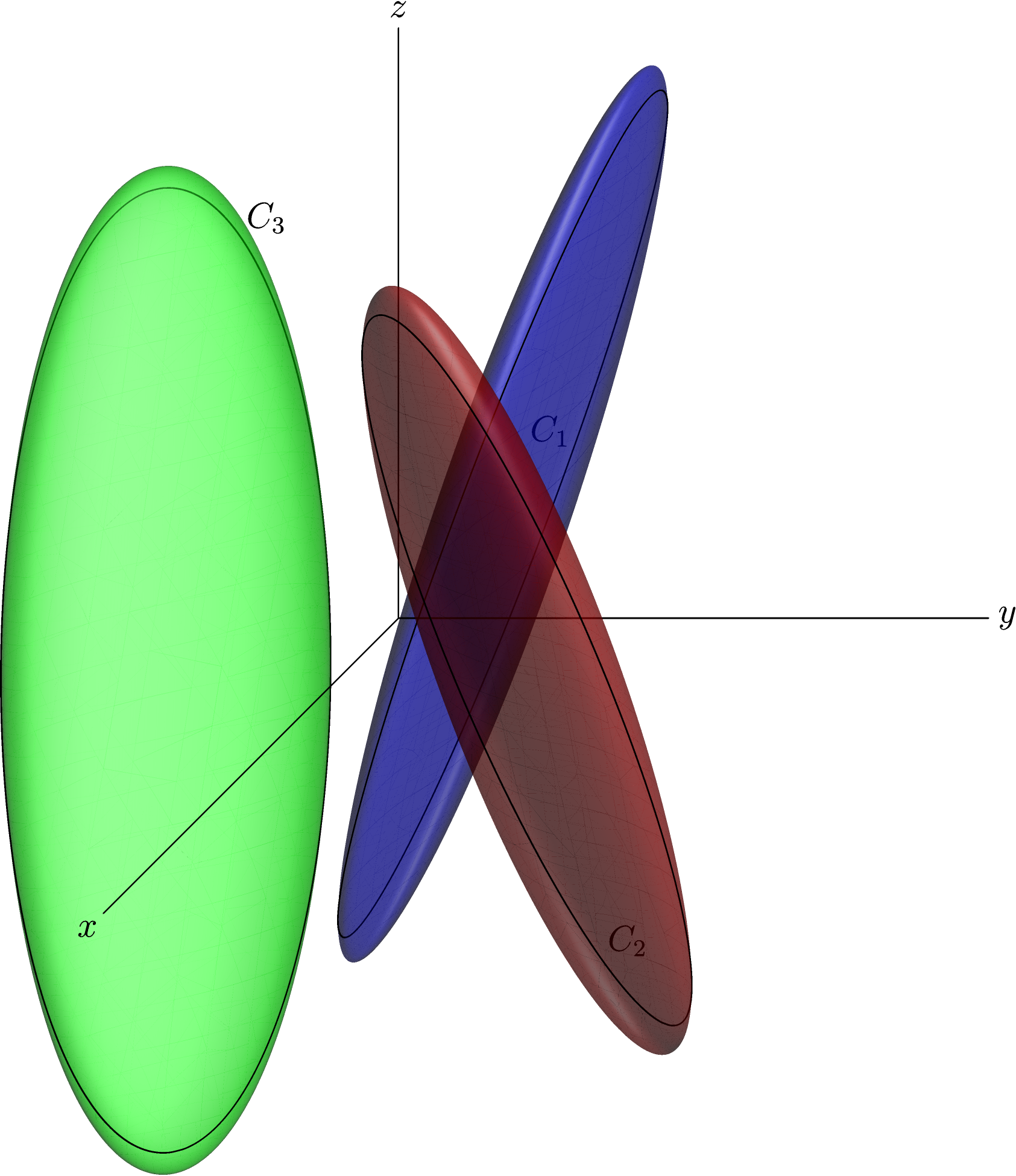}
  \caption{Isometric spheres of the $\iota_i$ reflections for $\mathscr{S}$.}
  \end{figure}

  Let $\iota_0,\iota_1,\iota_2$ the complex reflections induced by these complex chains, an straight computation shows that 
  \[\det\left(\frac{\partial \iota_i}{\partial (\zeta,v)}(t)\right)=\frac{|\sin^4(\theta)|}{12}.\] The transition matrix is given by 
  \[T=\frac{1}{12}\left(\begin{array}{ccc}
             0 & k(\theta)^4 & k(\theta)^4\\
             k(\theta)^4 & 0 & k(\theta)^4\\
             k(\theta)^4 & k(\theta)^4 & 0
            \end{array}\right)\]
   The eigenvalue of $T^\alpha$ have to satisfy $2\left(\frac{k(\theta)^4}{12}\right)^\alpha=1,$ so \[\alpha=\frac{\log(2)}{\log(12)-4\log(k(\theta))}.\] 

\begin{figure}[h]
 \includegraphics[scale=0.65]{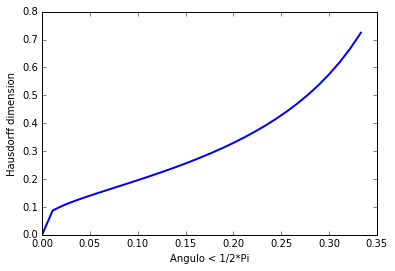}
 \caption{The computed Hausdorff dimension for $\Lambda_{CG}(\Gamma(\mathscr{S}))$ varying $\theta.$}
  \label{HausdorffComplexPants}
\end{figure}   
   As the action of a Schottky group defined by complex chains contained in $\heis\times\{0\}$ is like the classic Schottky group action on $\Hy^2,$ we can see the analog of the example computed in \cite{Mcmullen1998}, as you can see in Figure \ref{HausdorffComplexPants}.
\end{example}

\begin{example}[Non-Symmetric $\theta-$Schottky groups.]
We can put the generating chains' center in any part of the Heisenberg space; if we place these centers close to points in the standard finite $\R-$circle, more precisely on $\sec(\theta)(0,1),\sec(\theta)(0,-1)$ and $\sec(\theta)(-i,0),$ we can construct a $\theta-$Schottky group but instead of the angle belongs to $(0,\pi/3),$ we have to ask that the angle belongs to $(0,9\pi/40),$ this is to guarantee the convex cocompact property on the group. 

Since the standard finite $\R-$circle is a space circle whose planar projection is a lemniscate, we loose the symmetry on the chains and the centers. Let $\mathscr{C}$ denote the collection of the following three chains parametrized by

 \begin{eqnarray*}
  t&\mapsto& \left(\tan(\theta)e^{it},\sec^2(\theta)\right)\\
  t&\mapsto& \left(\tan(\theta)e^{it},-\sec^2(\theta)\right)\\
  t&\mapsto& \left(\tan(\theta)e^{it}-i\sec(\theta),-2\sec(\theta)\tan(\theta)\cos(t)\right)
 \end{eqnarray*}
 
 \begin{figure}[h]
\centering
\label{SpheresRCircle}
\includegraphics[scale=0.4]{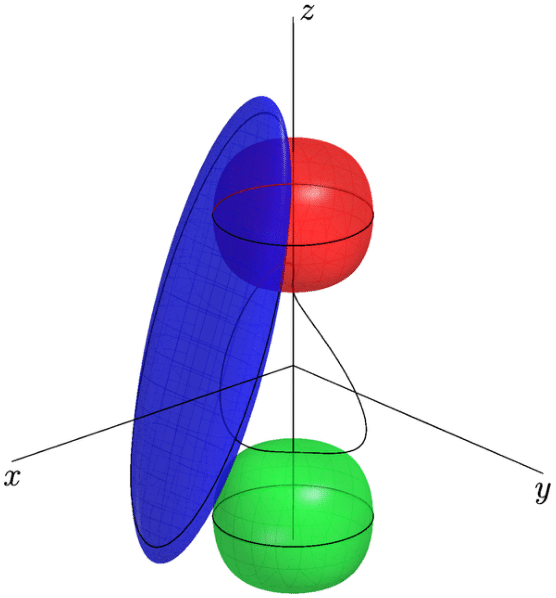}
\caption{Configuration of three finite chains that generate a $\theta-$Schottky Group}
 \end{figure}

 We denote by $\hat{\iota}_i$ the complex reflections generated by the previous chains, and a straight computations shows that the analysis done for the previous groups holds for these groups. For this reason if we took a set of uniformly distributed angles on $(0,9\pi/40)$ the Hausdorff dimension of its Chen-Greenberg limit set behave similarly as you can see in the following figure.
 
 \begin{figure}[h]
\centering
\label{HausdorffDimensionRPants}
\includegraphics[scale=0.4]{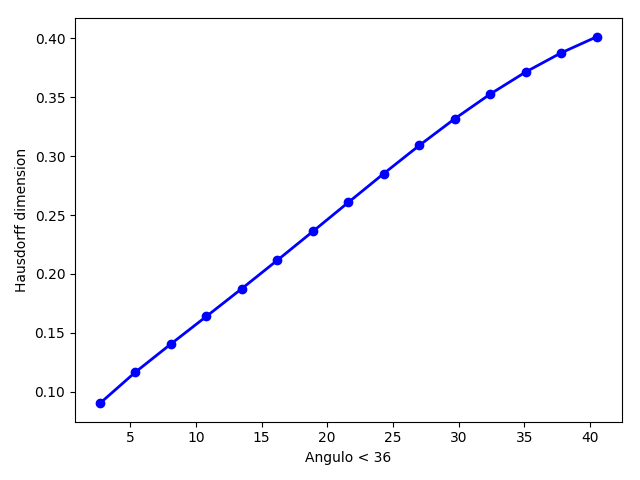}
 \caption{The computed Hausdorff dimension for $\Lambda_{CG}(\Gamma(\mathscr{C}))$ varying $\theta.$}
 \end{figure}
\end{example}

\appendix
\section{Pseudo-code for the McMullen Algorithm}
The algorithm was implemented in Python and is combination of the eigenvalue algorithm (see \cite{Mcmullen1998}) and the Newton method (see \cite{Garnett1988}). The algorithm is subdivided into function pieces, each of these pieces are used in the global function.

\begin{algorithm}[h]
 \SetKwInOut{Input}{Input}
 \SetKwInOut{Output}{Output}

    \underline{function} Inversion$(c,r,\zeta)$\;
    \Input{$c-$ array of the center of the complex chain. \\ $r-$ multiplier of the reflection, $|r|$ is the radious of the complex chain.\\ $\zeta-$ array of a point in Heisenberg different from $\infty.$}
    \Output{$z-$ array of a point Heisenberg coordinates.}
    $z[1]=c[1]-\frac{|r|^2(\zeta[1]-c[1])(|\zeta[1]-c[1]|^2+i(\zeta[2]-c[2]-2Im(\overline{\zeta[1]}c[1])))}{|\zeta[1]-c[1]|^4+(\zeta[2]-c[2]-2Im(\overline{\zeta[1]}c[1]))^2}$\;
    $z[2]=c[2]-\frac{|r|^4(\zeta[2]-c[2]-2Im(\overline{\zeta[1]}c[1]))+2|r|^2Im(\overline{-(\zeta[1]-c[1])(|\zeta[1]-c[1]|^2+i(\zeta[2]-c[2]-2Im(\overline{\zeta[1]}c[1]))\zeta[1])})}{|\zeta[1]-c[1]|^4+(\zeta[2]-c[2]-2Im(\overline{\zeta[1]}c[1]))^2}$\;
    \Return{$z$}
    \caption{Complex reflection defined by a complex chain in Heisenberg coordinates}
\end{algorithm}

\begin{algorithm}[h]
\label{DepthFirstComplex}
    \SetKwInOut{Input}{Input}
    \SetKwInOut{Output}{Output}

    \underline{function} words$(k,m,reflex)$\;
    \Input{$k-$ maximal depth of the Word tree, nonnegative integer\\ $m-$ number of reflections, nonnegative integer\\ $reflex-$ array of $m$ points in $\heis\times(\C\setminus\{0\})$.}
    \Output{$w-$ array of tagpoints\\ $wordsN-$ array of characters of words in the Tree of lenght $k-1.$}
    $N=m*(m-1)^{k-1}$\;
    $tagpoints=array([30000000])$\;
    $words=array([3000000])$\;
    \For{$i=1,\cdots,m$}{$tagpoints[i]=reflex[i][1]-\epsilon$ \tcp*{$\epsilon$ small}}
    $inv=array[1,2,\cdots,m]$\;
    $tag=zeros[3000000]$\;
    $num=zeros[k+2]$\;
    \For{$i=1,\cdots,m$}{$tag[i]=i$\; $words[i]=`i'$}
    $num[1]=1$\;
    $num[2]=m+1$\;
    \For{$lev=2,\cdots,k+2$}{$inew=num[lev]$\;
    \For{$j=1,\cdots,m$}{
    \For{$iold=num[lev-1],\cdots,num[lev]$}{
    \If{$j=inv[tag[iold]]$}{CONTINUE}
    \Else{$tagpoints[inew]=Inversion(reflex[j],tagpoints[iold])$\;
    $words[inew-1]=words[j]+words[iold]$\;
    $tag[inew]=j$\;
    $inew=+1$\;}}
    $num[lev]=inew$\;}}
    \For{$i=1,\cdots,num[k]-num[k-1]$}{$w[i]=tagpoints[num[k]]+i]$}
    \For{$i=1,\cdots,num[k-1]-num[k-2]$}{$wordsN[i]=words[num[k]+i]$}
    \Return{$w,wordsN$}
    \caption{Tagpoints and Words (see \cite{Mumford2002})}
\end{algorithm}

\begin{algorithm}[h]
 \SetKwInOut{Input}{Input}
 \SetKwInOut{Output}{Output}

 \underline{function} NewtonHausdorff$(d,T,\varepsilon)$\;
 \Input{$d-$ stimated value for the Hausdorff dimension (usual valor 1). \\ $T-$ a transition matrix. \\ $\varepsilon-$ desired error.}
 \Output{$dHauss-$ aproximated value for the Hausdorff dimension.}
 $N=Rank T$\;
 $aNew=d$\;
 $x=ones[N]/N$\;
 $x0=x$\;
 \While{$cont < 350$}{
 $Td=T^{aNew}$\;
 $a,x=PerronFrobenious(N,Td,x)$\;
 \If{$|a-1|<\varepsilon$}{
 BREAK\;}
 \Else{
 $d0=d+0.1$\;
 $TdE=T^{d0}$\;
 $aE,xE=PerronFrobenious(N,TdE,x0)$\;
 $Der=(aE-a)/0.01$\;
 $aNew=d+(1-a)/Der$\;
 $d=aNew$\;
 \If{$|d-1|<\varepsilon$}{BREAK}
 \Else{$cont=+1$}
 }
 }
 \Return{$d$}
 \caption{The Newton algorithm for the approximation of Hausdorff dimension}
\end{algorithm}

\bibliographystyle{amsplain} 
\bibliography{McMullenAlgorithmGeneralization.bib}
\end{document}